\documentclass[12pt]{amsart}
\usepackage{graphicx}
\usepackage{amssymb}

\usepackage{amsrefs}

 \usepackage{url}
\usepackage{booktabs}
\usepackage{hyperref}
\usepackage{verbatim}
\usepackage{xcolor}
\usepackage{array}

\usepackage[shortlabels]{enumitem}

\newcommand{\RN}[1]
    {\MakeUppercase{\romannumeral #1}}

\numberwithin{equation}{section}

\newtheorem{theorem}{Theorem}[section]
\newtheorem{lemma}[theorem]{Lemma}
\newtheorem{proposition}[theorem]{Proposition}
\newtheorem{corollary}[theorem]{Corollary}

\newtheorem{remark}[theorem]{Remark}
\newtheorem{definition}[theorem]{Definition}
\newtheorem{notation}[theorem]{Notation}
\newtheorem{example}[theorem]{Example}

\DeclareMathOperator{\Bl}{Bl}

\newcommand{\PP}[0]{\mathbb{P}}

\newcommand{\R}[0]{\mathbb{R}}

\newcommand{\Q}[0]{\mathbb{Q}}

\DeclareMathOperator{\Bs}{Bs}

\DeclareMathOperator{\Eff}{Eff}
\DeclareMathOperator{\Nef}{Nef}
\DeclareMathOperator{\Mov}{Mov}
\def\Pic{\operatorname{Pic}}
\DeclareMathOperator{\N}{N}
\DeclareMathOperator{\A}{A}
\DeclareMathOperator{\NE}{NE}

\newcommand{\OO}{\mathcal{O}}
    
\title{Duality  and polyhedrality of cones for Mori dream spaces}

\author[Brambilla]{Maria Chiara Brambilla}\address{
Universit\`a Politecnica delle Marche, via Brecce Bianche, I-60131 Ancona, Italia}
\email{m.c.brambilla@univpm.it}

\author[Dumitrescu]{Olivia Dumitrescu}\address{
1. University of North Carolina at Chapel Hill, 
340 Phillips Hall CB 3250 NC 27599-3250 and 2. Max-Planck Institute for Mathematics, Bonn, Vivatsgasse 7
53111 Bonn, Germany}
\email{dolivia@unc.edu}

\author[Postinghel]{Elisa Postinghel}\address{
Dipartimento di Matematica, Universit\`a degli Studi di Trento, via Sommarive 14
I-38123 Povo di Trento (TN)}
\email{elisa.postinghel@unitn.it}

\author[Santana Sanchez]{Luis Jos\'e Santana S\'anchez}\address{Departamento de Matem\'aticas, Estad\'istica e I. O., and IMAULL, Universidad de La Laguna 38200, La Laguna, Tenerife, Espa\~na}\email{lsantans@ull.edu.es}


\begin{document}

\begin{abstract} 
Our goal is twofold. On one hand, we show that the cones of divisors ample in codimension $k$ on a Mori dream space are rational polyhedral. 
On the other hand, we study the duality between such cones and the cones of $k$-moving curves by means of the Mori chamber decomposition of the former. We give a new proof of the weak duality property (already proved by Payne and Choi) and we exhibit an interesting family of examples for which  strong duality holds.
\end{abstract}

\maketitle

\section{Introduction}

Let $X$ be a normal $\mathbb Q$-factorial variety of dimension $n$ and with  irregularity zero, defined over an algebraically closed field of characteristic zero.

For any integer $0 \le k \le n-1$, we consider the cone {$D_k$} of numerical classes of divisors on $X$ whose stable base locus has codimension larger than $k$ and we set $\mathcal{D}_k$ to be the closure of $D_k$ in the N\'eron-Severi space $\N^1(X)_\mathbb{R}$. Following \cite{Payne}, we call $\mathcal D_k$ the  cone of divisors that are \emph{ample in codimension $k$}. 
{Clearly $\mathcal{D}_0$ is the pseudo-effective cone $\overline{\Eff(X)}$, $\mathcal{D}_1$ is the cone of movable divisors, and
 $\mathcal D_{n-1}$ is the nef cone by Kleiman's theorem.
When $X$ is a Mori dream space, the cones $\mathcal{D}_0$, $\mathcal{D}_1$ and $\mathcal{D}_{n-1}$  are rational polyhedral, as shown in \cite{Hu-Keel}, moreover $\mathcal{D}_0$ and $\mathcal{D}_1$ admit a finite \emph{Mori chamber decomposition} of which $\mathcal{D}_{n-1}$ is a chamber.
In this article we prove that each  cone $\mathcal{D}_k$ inherits the chamber decomposition and in particular   it is rational polyhedral (and $D_k$ is closed).

We have the obvious filtration of cones
$$\Nef(X)=\mathcal{D}_{n-1}\subseteq\cdots\subseteq \mathcal{D}_1\subseteq\mathcal{D}_0=\overline{\Eff(X)};$$
on the other hand we 
shall consider an analogous filtration of cones of pseudo-effective $1$-cycles:
$$\overline{\NE(X)}=\mathcal{C}_{n-1}\supseteq\cdots\supseteq \mathcal{C}_1\supseteq\mathcal{C}_0,$$
where $\mathcal{C}_k$ is the closure of the cone in $\N_1(X)_\mathbb{R}$ 
generated by classes of $k$-\emph{moving} curves, i.e. classes of irreducible curves moving in a family that sweeps out an $(n-k)$-dimensional subvariety of X.
Corresponding ends of the above filtrations are dual to each other with respect to the standard intersection pairing. Indeed the dual of the Mori cone $\mathcal{C}_{n-1}^\vee$  is  the nef cone and  {the cone $\mathcal{C}_0$ is dual to
the pseudo-effective cone $\overline{\Eff(X)}$, as shown in \cite{BDPP}.} 

{Since}
 duality of the cones $\mathcal{D}_k$ and $\mathcal{C}_k$, for $1\le k\le n-2$, is not to be expected in general,
{it is a  natural question to ask} under which conditions and for which  integers $1\le k\le n-2$  
{{\it strong duality} holds, that is} $\mathcal{D}_k^\vee=\mathcal{C}_k$.

In this paper, we show that the answer is positive, for every $k$,  for all blow-ups of $\mathbb{P}^n$ at $s$ points in general position that are Mori dream spaces.

The answer to the duality question  is not positive in general but a weaker version holds, as proved in \cite{Payne} for toric varieties and in \cite{ChoiCONES} for Mori dream spaces: for a Mori dream space $X$,  $\mathcal{D}_k^\vee$ is generated by curves that move in families sweeping out the birational image, on a small modification of $X$, of $(n-k)$-dimensional subvarieties of X. In other terms, $k$-moving curves on $X$ are not enough to reconstruct, by duality, the cone $\mathcal{D}_k$; instead one needs to consider $k$-moving curves on all small modifications of $X$, which are of finite number.
We will call this property \emph{weak duality}.

The first main theorem of this paper, Theorem \ref{MCD-cones-divisors}, describes the combinatorial properties of the cones $\mathcal D_k$ for Mori dream spaces,  providing a different perspective on the study of these cones and setting the bases for a new proof of the weak duality statement.
As a crucial tool we use the stable base locus decomposition of the effective and movable cones,  which is refined by the Mori chamber decomposition, and we prove that it is compatible with the filtration of cones $\mathcal{D}_k$.
 \begin{theorem}
 \label{intro-thm1}
If $X$ is a Mori dream space, then the cones $\mathcal{D}_k$  are unions of Mori chambers, hence  rational polyhedral, and the cones $D_k$ are closed, i.e. $D_k=\mathcal{D}_k$, for every $0\leq k\leq n-1$.
\end{theorem}
The fact that the cones $\mathcal{D}_k$ are rational polyhedral in the case of toric varieties was proved by Payne in \cite{Payne} where, furthermore,  a constructive description was presented  in terms of the boundary divisors.

Choi asked if the $\mathcal D_k$'s are rational polyhedral also for Mori dream spaces, see \cite [Question 4.4]{ChoiCONES} and \cite[p. 1177]{Choi14}. 
The first statement of  Theorem \ref{intro-thm1} answers affirmatively the question for divisors on Mori dream spaces.
Moreover if we restrict to the toric case, it gives  another proof via a new approach to the rational polyhedrality of the cones $\mathcal{D}_k$.

Our second main result, proved in Theorems \ref{strong-duality-thm-n+3} and \ref{strong-duality-thm-X^4_8}, provides a first family of examples for which strong duality holds.

 \begin{theorem}
 \label{intro-thm2}
	The strong duality holds for all Mori dream blow-ups of projective spaces at points in general position. 
\end{theorem} 

We remark that the geometry of these spaces and their chamber decompositions are completely encoded by 
 \emph{Weyl cycles}. We will investigate more deeply
the interplay between the birational geometry and the action of the Weyl group
in forthcoming work, continuing the study started in \cite{BraDumPos3,BDP-Ciro}.

Moreover, in future work we will seek   criteria for  strong duality to hold for general Mori dream spaces.

Finally, we point out that ampleness in codimension $k$ relates to different notions of cohomological ampleness for divisors. These notions were studied by Andreotti-Grauert and Sommese (naive ampleness), by Demailly-Peternell-Schneider (uniform ampleness) and by Totaro ($T-$ampleness),  
see \cite{Totaro} and references therein. 
In \cite[Question 11.2]{Totaro}, Totaro asked if, for Fano varieties in characteristic zero,
the  closure  of the cone of cohomologically $k-$ample  divisors is a union of rational polyhedral cones.
In \cite{BOP-S} the authors proved that this property holds for simplicial toric varieties.

The article is organised as follows. 
In Section \ref{section-preliminaries} we recall some preliminaries on cones of curves and divisors and on chamber decompositions.
Section \ref{sect-geometry-cones} is dedicated to our first main result, Theorem \ref{intro-thm1}, that establishes the polyhedrality of the cones $\mathcal D_k$ for all Mori dream spaces. Section \ref{sectionb-weak-duality} contains our proof of the weak duality property for $\mathcal D_k$ and the cones of curves $\mathcal C_k$ for Mori dream spaces,  by means of the Mori chamber decomposition of the effective cone. We also include the discussion of a toric example for which strong duality does not hold. 
In Section \ref{section-strongduality-P^n} we prove our second main result, Theorem \ref{intro-thm2}, which concerns a family of examples of Mori dream spaces satisfying strong duality.

\subsection*{Acknowledgements} 
We would like to thank Artie Prendergast-Smith for various useful comments.
We thank the referee for several useful comments.

The first and third authors are members of INdAM-GNSAGA. 
The second author would like to express her gratitude to the Max-Planck Institute for Mathematics, Bonn, for their generosity, hospitality and stimulating environment. During the time of this project, she was also supported by the NSF-FRG grant DMS 2152130 and the Simons Foundation collaboration grant 855897. 
The third author's research is funded by the European Union under the Next Generation EU PRIN 2022 \emph{Birational geometry of moduli spaces and special varieties}, Prot. n. 20223B5S8L. 
The fourth author was partially supported by the ULL funded research project MACACO.

\section{Preliminaries}\label{section-preliminaries}
Throughout the paper, we assume that  $X$ is a normal $\mathbb Q$-factorial variety of dimension $n$, defined over an algebraically closed field of characteristic zero.

In this section we recall some preliminaries. As a general reference see \cite{Laz1}.
\subsection{Rational polyhedral cones}
Let $V\subset\R^n$ be a convex cone. A subcone $W\subset V$ is {\it extremal} if $u+v\in W$ implies $u,v\in W$. A one-dimensional subcone is called a {\it ray}. 
 A cone is {\it polyhedral} if it is the intersection of a finite number of half-spaces through the origin. 
By the Minkowski-Weyl theorem, a cone is  polyhedral if and only if it is finite, i.e.\ it is the convex conical hull of a finite number of vectors. In particular a polyhedral cone is closed.
A convex cone is {\it rational polyhedral} if it  is generated by a finite set of integer vectors.

\subsection{Curve classes and divisor classes}
Let $\textrm{Pic}(X)$ be the Picard group of $X$ and let $\A_k(X)$ be the group of $k$-cycles on $X$, for $1\le k\le n-1$.
We denote with 
$\N^1(X)_\R=(\Pic(X)/\equiv)\otimes \R$ the N\'eron-Severi space of $X$ and with  $\N_1(X)_\R=(\A_1(X)/\equiv)\otimes \R$ its dual space with respect to the standard intersection pairing.
The cones of divisors $\Nef(X)\subseteq \overline{\Eff(X)}\subset\N^1(X)_\R$ are, respectively, the nef and the pseudo-effective cones. The former is dual to the Mori cone $\overline{\NE(X)}\subset \N_1(X)_\R$, which is the closure of the cone of effective $1$-cycles; the latter is dual to the closure of the cone of moving curves, that are numerical curve classes such that the irreducible representatives of the class pass through a general point of $X$, see \cite[Theorem 0.2]{BDPP} or \cite[11.4.C]{Laz2}.

\subsection{Stable Base Locus Decomposition}
Given an effective divisor $D$ on $X$, we denote by $\Bs(D)$ the base locus of the linear system $|D|$. We recall, from \cite[Definition 2.1.20]{Laz1},
that
  the stable base locus of $D$ is the Zariski closed set
$$\mathbb{B}(D) = \bigcap_{m>0} \Bs(mD).$$
Under the assumption $q(X)=h^1(X,\OO_X)=0$, linear and numerical equivalence of $\Q$-divisors coincide.
Thus the stable base locus is well defined for  numerical equivalence classes of $\Q$-divisors on $X$, \cite{ELMNP}.

The {\it Stable base locus decomposition} is a wall-and-chamber decomposition of the pseudo-effective cone $\overline{\Eff(X)}$ of $X$, where the stable base locus is constant in the interior of the chambers, see e.g. \cite[Sect 4.1.3]{Huizenga} and \cite{LMR}  for more details.

We will call {\it stable base locus subvariety of $X$} any irreducible and reduced subvariety $Y\subset X$ that whenever it is a component of the base locus of an effective divisor on $X$, it is a component of its stable base locus.

A divisor $D$ in $X$ is called {\it movable} if its stable base locus has codimension at least two in $X$. 
The movable cone of $X$ is the convex cone $\Mov(X)$ generated by classes of movable divisors.  We have the following obvious inclusions of cones: $\Nef(X)\subseteq\overline{\Mov(X)}\subseteq \overline{\Eff(X)}$.

\subsection{Mori dream spaces and Mori Chamber Decomposition}\label{section-prelim-MCD}
We recall the  definition of \emph{Mori dream space} from 
\cite[Definition 1.10]{Hu-Keel}.
Mori dream spaces include several classes of interesting varieties such as toric varieties, $\mathbb{Q}$-factorial spherical varieties \cite{Brion}, Fano varieties and varieties of Fano type \cite{BCHM} among others.
 
\begin{definition}
 A normal projective $\mathbb{Q}$-factorial variety $X$ is called a Mori dream space if the following conditions hold:
 \begin{enumerate}
\item $\Pic(X)$ is finitely generated, that is $h^1(X,\OO_X)=0$,
\item $\Nef(X)$ is generated by the classes of finitely many semi-ample divisors,
\item there is a finite collection of small $\mathbb{Q}$-factorial modifications $f_i :S\dashrightarrow X_i$, such that each $X_i$ satisfies (2), and $$\Mov(X)=\bigcup f_i^*(\Nef(X_i)).$$ 
\end{enumerate}
\end{definition}

Recall that a small $\mathbb{Q}$-factorial modification (SQM) $f:X\dasharrow X'$ is a contracting birational map that is an isomorphism in codimension $1$, with $X'$ a normal projective and $\mathbb{Q}$-factorial variety. If $X$ is a Mori dream space, then every SQM of $X$ factors as a finite sequence of flips and a regular contraction.

If $X$ is a Mori dream space, then the effective cone $\Eff(X)$ is rational polyhedral, and in particular closed, see e.g. \cite{okawa}.
The section ring $R(X,D)$ of a divisor $D$ on $X$ is finitely generated and it induces a natural rational map $$f_D:X\dashrightarrow X_D\subseteq\textrm{Proj}(R(X,D)),$$ which is regular away from $\mathbb{B}(D)$. Two $\mathbb{Q}$-divisors $D_1$ and $D_2$ are \emph{Mori equivalent} if $f_{D_1}$ and $f_{D_2}$ have the same Stein factorisation. Given an effective divisor $D$ on $X$, the \emph{Mori chamber} of $D$ is the closure of the set of all divisors that are Mori equivalent to $D$. Mori chambers are convex cones and the set of such cones form a fan, the \emph{Mori fan} of $X$.  The maximal cones of this fan form a wall-and-chamber decomposition of the effective cone, called the \emph{Mori chamber decomposition}. Such decomposition 
 is a refinement of the stable base locus decomposition, see e.g.  \cite{Hu-Keel,okawa, LMR} for details.}

\subsection{Cones of $k$-moving curves}
We recall the notion of $k$-moving curve, for $0\le k\le n-1$, cf. \cite{Payne}.

\begin{definition}
An irreducible curve $C\subset X$ is {\em $k$-moving} if it belongs to an algebraic family of curves, whose  irreducible elements cover a Zariski open subset of an effective cycle of dimension (at least) $n-k$. 

We define 
$C_k$ to be the cone generated by the classes of $k$-moving curves in  $\N_1(X)_\R$ and we set $\mathcal{C}_k=\overline C_k$ to be the closure in $\N_1(X)_\R$.
\end{definition}

\begin{example} \label{lin curve} 
Let $p_1,\dots,p_s\in \PP^n$ be points in general position and let $X^n_s$ be to the blow of $\PP^n$ at these points. For every three distinct indices $i,j,k$, the strict transform on $X^n_s$ of an irreducible conic through $p_i,p_j,p_k$ sweeps out the strict transform of the linear span of the points and hence it is an $(n-2)-$moving curve.
Similarly, the strict transform of 
an irreducible curve $C$ of degree $r$ through $r+1$ points
is an $(n-r)-$moving curve since it sweeps out the strict transform of the $r-$dimensional linear span of the points.
\end{example}

It follows from the definitions that a $k$-moving curve class is also $(k+1)$-moving, for $0\le k\le n-2$.
We  obtain the following filtration of cones of curves:
\begin{equation}\label{filtration-cones-curves}
\overline{\NE(X)} = \mathcal{C}_{n-1} \hookleftarrow \mathcal{C}_{n-2}\hookleftarrow \cdots \hookleftarrow \mathcal{C}_1\hookleftarrow \mathcal{C}_0.
\end{equation}
Clearly
$\mathcal C_{n-1}=\overline{\NE(X)}$ is the Mori cone of curves (dual to $\Nef(X)$). Furthermore,
$\mathcal C_0$ is the closed cone of the moving curve classes (dual to $\overline{\Eff(X)}$).

\subsection{Cones $\mathcal D_k$ of divisors that are ample in codimension $k$}
There is a stratification of the pseudoeffective cone of divisors as follows.

\begin{definition}\label{Dk cone definition}
Let $0\le k \le n-1$.
We define $D_k$ to be the cone generated by the classes of {effective} divisors  with no component of the stable base locus of dimension $\ge n-k$;  we denote with
$\mathcal D_k$ the closure of $D_k$ in $\overline{\Eff}(X)$.
\end{definition}
Using the same terminology as Payne \cite{Payne}, we call such divisors {\it ample in codimension $k$}, indeed $\mathcal D_k$ is also the closure of the cone of numerical classes of divisors D such that there is an open set U, with $\OO(D)_{|U}$ ample, whose complement has codimension greater than $k$.
We have the following obvious inclusions:
\begin{equation}\label{filtration-cones-divisors}
\Nef(X)=\mathcal{D}_{n-1} \hookrightarrow \mathcal{D}_{n-2} \hookrightarrow \cdots \hookrightarrow \mathcal{D}_{1} \hookrightarrow \mathcal{D}_{0}=\overline{\Eff(X)}.
\end{equation}

Recall that if $X$ is a Mori dream space, then $\Eff(X)$ is rational polyhedral, hence $D_0=\mathcal D_0=\Eff(X)$. Moreover,  also 
$\mathcal D_{n-1}=\Nef(X)$ and 
$\mathcal{D}_{1}=\Mov(X)$
are rational polyhedral.
 Choi posed the natural question as to whether  all cones $\mathcal D_k$ are rational polyhedral for Mori dream spaces, see \cite[Question 4.4]{ChoiCONES} or \cite[page 1177]{Choi14}. In Theorem  \ref{MCD-cones-divisors} below (cf. Theorem \ref{intro-thm1}), we answer this question affirmatively
 and we also prove that all the cones $D_k$ are closed.

This  is not true in general for non-Mori dream spaces, as the following example shows for $k=n-1$.

\begin{example}
Let $X=\Bl_9(\PP^2)$ be the blow up of the projective plane at nine general points, which is not a Mori dream space. 
Since the dimension of $X$ is $2$, we have $\mathcal{D}_1=\Nef(X)$ 
and it is well-known that the nef cone is not finitely generated. Hence $\mathcal D_1$ is not rational polyhedral.

Now we show that $D_1\neq\mathcal{D}_1$.
Indeed the anticanonical curve $F$, which is the strict transform on $X$ of a plane cubic through the nine points, is nef,  that is $F\in\mathcal{D}_1$.
In fact the Mori cone of $X$ is generated by the $(-1)$-curves and by $F$, and we have $F^2=0$ and $F\cdot E=1$ for every $(-1)$-curve $E$,  see \cite{nagata}.

However $F\not\in D_1$. In fact, since there is a unique plane cubic through nine points, the linear system $|F|$ has dimension $0$. Even more:  $|mF|=\{mF\}$, for every integer $m\ge1$, and this goes back to the classical work of Castelnuovo \cite{Castelnuovo91}.  This implies that $F$ is the stable base locus of its own linear system $|F|$.
\end{example}

\section{The geometry of the cones $\mathcal D_k$}\label{sect-geometry-cones}
In this section we show that the stable base locus decomposition of the pseudo-effective cone of a normal $\mathbb Q$-factorial projective variety $X$ with irrationality zero  is compatible with the filtration of cones \eqref{filtration-cones-divisors}, that is, every cone $\mathcal{D}_{k}$ decomposes as a union of stable base locus chambers.
Moreover 
if $X$ is a Mori dream space, since each stable base locus chamber is a union of Mori chambers, we obtain that $\mathcal{D}_{k}$ is a union of such chambers. Furthermore, it is possible to characterise the set of Mori chambers whose union is $\mathcal{D}_{k}$, for every $1\le k\le n-1$.

\begin{lemma}\label{filtration+SBLD} 
Let $X$ be a normal 
 $\mathbb Q$-factorial 
projective variety and assume $h^1(X,\mathcal{O}_X)=0$. Then, 
for every $k$, the cone ${D}_{k}$ is a union of stable base locus chambers. 
\end{lemma}
\begin{proof}
We shall prove that the stable base locus chamber decomposition of the pseudo-effective cone of $X$ induces a chamber decomposition on each cone $D_k$, by showing that every stable base locus chamber is contained in the difference set $D_{k-1}\setminus D_k$, for some $1\leq k  \leq n-1 $.

In fact, let $B$ be a stable base locus chamber, and let $D$ be an effective divisor that lies in $B$.
Let $n-k$ be the dimension of $\mathbb{B}(D)$ or, equivalently, 
the maximum dimension of the components of $\mathbb{B}(D)$. 
The inequality  $\dim\mathbb{B}(D)\ge n-k$ implies that $D$ does not lie in ${D}_k$, while since $\dim\mathbb{B}(D)\le n-k$, $D$ lies in ${D}_{k-1}$, that is, $D \in D_{k-1}\setminus D_k$. Since $B$ is a stable base locus chamber, by definition the same holds for any other divisor in $B$. Thus, $B\subset D_{k-1}\setminus D_k$.
\end{proof}

In the next theorem, 
 we propose a detailed description of the Mori chamber decomposition of the cones $\mathcal{D}_k$ for  Mori dream spaces.

\begin{notation} 
We set $F_1:=\{\textrm{Id}:X\to X\}$.
Let $2\le r\le n-1$ and let ${F}_{r}$ be the set of all small $\mathbb{Q}$-factorial modifications $f_i:X\dashrightarrow X_i$  given by sequences of  flips of subvarieties whose birational preimage on $X$ has dimension $< r$. 
\end{notation}

\begin{theorem}\label{MCD-cones-divisors}
Let $X$ be a Mori dream space and fix $1\le k\le n-1$.
\begin{enumerate}
\item 
The cone $\mathcal{D}_k$  is rational polyhedral and the following  holds:
$${\mathcal D_{k}}= \bigcup_{f_i\in {F}_{n-k}} {f_i^*}(\Nef(X_i)).$$
\item The cone $D_k$ is closed, that is: $D_k=\mathcal{D}_k$.
\end{enumerate}
\end{theorem}
\begin{proof}
Recall that, for a Mori dream space, every closed stable base locus chamber is a finite union of Mori chambers, see  Section \ref{section-prelim-MCD}.
By {Lemma \ref{filtration+SBLD},  $D_{k}$} is a union of stable base locus chambers and hence its closure $\mathcal D_k$ is a finite union of Mori chambers.
We shall now characterise which are the Mori chambers appearing in the union as follows.
Let $f_i:X \dashrightarrow X_i$ be a small $\mathbb{Q}$-factorial modification of $X$ and let $f_i^\ast(\Nef(X_i))$ be the corresponding Mori chamber. Let $D\in f_i^\ast(\Nef(X_i))$ be any  pull-back of an ample class on $X_i$: we have that $D$ is movable on $X$.
The stable base locus of  $D$ equals the indeterminacy locus of $f_i$ (see e.g. \cite{Hu-Keel, okawa}). 
Hence, the latter is a union of subvarieties of $X$ of dimension $< n-k$ if and only if $D$ lies in  $D_{k}$. This argument is independent of $D$ but rather depends on its base locus, and it applies to any pull-back on $X$ of an ample divisors on $X_i$. Therefore, the interior of the Mori chamber  $f_i^\ast(\Nef(X_i))$ is contained in $D_k$.  Finally, taking closures, we have that  $f_i^\ast(\Nef(X_i))\subseteq \mathcal{D}_{k}$ if any only if $f_i\in F_{n-k}$.

Since $X$ is a Mori dream space, there are finitely many Mori chambers and each is a rational polyhedral cone. Therefore the 
above argument shows that for every $k$, the convex cone $\mathcal{D}_k$ is a union of finitely many rational polyhedral cones. Since $\mathcal{D}_k$ is a cone,  it is a rational polyhedral cone. This proves (1).

We now prove (2), i.e. that $\mathcal{D}_k=D_k$. In order to do so, we show that every extremal ray of $\mathcal D_k$ lies in $D_k$. Let $E$ be an effective divisor on $X$ spanning an extremal ray. Then $E$ spans an extremal ray of some Mori chamber $f_j^\ast(\Nef(X_j))\subseteq \mathcal D_k$ and hence the pull-back of a semi-ample divisor $E_j$ on $X_j$, where $f_j:X\dasharrow X_j$, with $f_j\in F_{n-k}$, by (1).
Let $m\ge1$ be an integer such that $mE_j$ is basepoint free.  The stable base locus of $mE=f_j^\ast(mE_j)$  is therefore the indeterminacy locus of $f_j$ which is a union of subvarieties of $X$ of dimension strictly less than $n-k$. Since $\mathbb{B}(E)=\mathbb{B}(mE)$,  we conclude that $E$ lies in $D_k$.

\end{proof}
\begin{remark} For $k=1$, Theorem \ref{MCD-cones-divisors} is just property (3) of the definition of Mori dream spaces. 
For $k=n-1$ it is always true {by Kleiman's theorem}.  \end{remark}

For non-Mori dream spaces, none of the statements of Theorem \ref{MCD-cones-divisors} hold in general; in fact the properties may hold or fail as we vary $k$. The following is an instance of this.

\begin{example}
Let  $X=X^5_9$ be the blow-up of $\PP^5$ at nine points in general position. Recall that $X$ is not a Mori dream space, in particular $\mathcal D_{0}=\overline{\Eff(X)}$ is not finitely generated.
On the other hand, the Mori cone $\mathcal C_4$ is finitely generated by lines, see e.g. \cite[Theorem 3.2]{DP-positivityI}, hence  its dual, $\Nef(X)$,  is rational polyhedral. 
Moreover the following holds: $\mathcal D_{4}=D_4=\Nef(X)$.  Indeed, the nef cone $\Nef(X)$ is generated by the following classes: $H, H-E_i, 2H-\sum_{i\in I}E_i$, for every $3\le |I|\le 9$ as one can easily check with the help of any computer software that can compute cone duals, such as \emph{Magma}  \cite{Magma}. It is an easy exercise to show that each of them is basepoint free, hence it lies in $D_4$.

It would be interesting to determine which, if any,  of the intermediate cones $\mathcal D_{1}$, $\mathcal D_{2}$ and $\mathcal D_{3}$ are  finitely generated.
\end{example}

\section{Weak duality}\label{sectionb-weak-duality}
Duality of the cones $\mathcal D_k$ and $\mathcal C_k$ is not to be expected in general, but a weaker version for Mori dream spaces holds, that takes into account $k$-moving curves not just on $X$ but also on its small modifications. We will refer to this as \emph{weak duality}, as opposed to \emph{strong duality} that holds whenever $\mathcal C_k=\mathcal D_k^{\vee}$, for every $k$. 

It is easy to show that $\Mov(X)\subseteq \mathcal C_1^\vee$. In fact if $D$ is in the dual of $\mathcal C_1$ then $D\cdot C\ge0$ for every $1$-moving curve $C$. If it is not in the dual, then there is a $1$-moving curve $C$ such that $D\cdot C<0$, and so the divisor ($(n-1)$-cycle) spanned by the algebraic family of $C$ is a fixed component of the base locus of $D$. 
If $n=2$, then it holds that $\Mov(X)=\mathcal C_1^\vee$. In fact the first cone is the nef cone, the second cone is the Mori cone.

A similar inclusion of cones is valid in general.
\begin{remark}\label{easy-rmk}
If $C\in\mathcal C_k$, we know that $C$  moves in a  family that sweeps out an effective cycle $V$ of dimension  $n-k$. If $C\not\in \mathcal D^\vee_k$, then there exists an effective  divisor $D$ such that $C\cdot D<0$ and 
whose stable base locus does not have any component of dimension $\ge n-k$. But then $V$ is in the {stable} base locus of $D$ which is a contradiction.  This proves that 
$$\mathcal{C}_k\subseteq \mathcal D_k^\vee.$$
\end{remark}
We adopt the definition of $k$-moving curves on small modifications proposed by Payne in \cite{Payne}.

\begin{definition}
 We say that a curve class $c\in \N_1(X)_\R$ is 
 \emph{bir-$k$-moving} on $X$ if 
there exists a 
small modification $f_i:X \dashrightarrow X_i$, such that there is a curve
$C\subset X_i$, in the class corresponding to $c$,  that moves in a family which sweeps out a cycle $f_i(V)$ birational to a cycle $V\subset X$ of codimension  $k$. 

We denote by $\mathcal{C}_k^{bir}$ the {closure of the} cone generated in $\N_1(X)_{\R}$ by all the 
 bir-$k$-moving classes on $X$.
\end{definition}

In \cite{Choi-Gongyo}, a version of  the cone theorem for $k$-moving curves on small modifications of $X$ is discussed, in the setting of the  minimal model program.

Applying the same idea as in Remark \ref{easy-rmk}, one can easily see that the inclusion $\mathcal{C}^{bir}_k\subseteq \mathcal D_k^\vee$ holds.
Using the results of Section \ref{sect-geometry-cones}, we prove that the latter inequality holds with equality, hence
 recovering a result that was first proved by Payne 
for complete $\mathbb{Q}-$factorial toric varieties \cite{Payne} and later generalised by Choi for log Fano varieties and for Mori dream spaces \cite[Corollary 4.3]{ChoiCONES}, running the log minimal model program. 

The main ingredient of our proof is the Mori chamber decomposition of the effective cone of $X$, and of $\mathcal{D}_k$,  proved in Theorem \ref{MCD-cones-divisors}.

\begin{theorem}[{Weak duality theorem for Mori dream spaces}]\label{PayneChoi}
If $X$ is a  Mori dream space then weak duality holds, that is 
 $$\mathcal{C}_k^{bir}=\mathcal D_k^\vee.$$
\end{theorem} 
\begin{proof}

 We only have to prove that the following inclusion holds: $\mathcal D_k^\vee\subseteq \mathcal{C}^{bir}_k.$ 
Every facet $\mathcal F$ of $\mathcal D_k$ corresponds, by duality, to an extremal ray of $\mathcal D_k^\vee$, spanned by a curve class $c=c(\mathcal{F})\in \overline{\textrm{NE}(X)}$. We shall prove that  $c\in \mathcal{C}^{bir}_k$. We consider the following cases separately.

Assume first that  $\mathcal{F}$ is a facet of $\mathcal D_0=\overline{\textrm{Eff}(X)}$. Then $c$ is a moving curve class and the statement follows.

Now, let $\mathcal F$ be a facet of $\mathcal D_k$ but not a facet of $\mathcal D_0$. 
Then there is a Mori chamber inside $\mathcal D_k$ and adjacent to $\mathcal F$, that means that a facet of the chamber is contained or is equal to $\mathcal F$. Let $f_-^*(\Nef(X_-))\subset \mathcal D_k$ be one such chamber, for $f_-:X\dashrightarrow X_-$ a  small modification of $X$. It follows that $c\in \N_1(X)_\mathbb{R}\cong \N_1(X_-)_\mathbb{R}$ generates an extremal ray of the Mori cone of $X_-$.
Let $V_-\subset X_-$ be a subvariety of $X_-$ of maximal dimension swept out by a family of irreducible curves of class $c$. 

Now, let $f_+^*(\Nef(X_+))\ast \textrm{ex}(f_+)$ be a Mori chamber  outside $\mathcal D_k$, see \cite[Proposition 1.11]{Hu-Keel}, 
and adjacent to $\mathcal F$, that has a common facet with $f_-^*(\Nef(X_-))$ contained in $\mathcal F$. The contracting birational map $f_+:X\dashrightarrow X_+$ can be decomposed as $f_+=g\circ f_-$, where $g: X_-\dashrightarrow X_+$ corresponds to crossing the wall $\mathcal F$, and the indeterminacy locus of $g$ is $V_-$.
Take $D\in f_+^*(\Nef(X^+))\ast \textrm{ex}(f_+)$ a divisor on $X$ whose ray lies in the relative interior of the Mori chamber. 
Since $D\notin \mathcal{D}_k$, then $D$ must have a component of its stable base locus of codimension $\le k$. But the stable base locus of $D$, as a divisor on $X$, is the indeterminacy locus of $f_+$, which is the union of the indeterminacy locus of $f_-$ and $V:=f_-^{-1}(V_-)$.
Since by Theorem \ref{MCD-cones-divisors}, the indeterminacy locus of $f_-$ has codimension $>k$, then $V$ must have codimension $p\le k$.
In particular,  $V$ is not in the indeterminacy locus of $f_-$ and therefore $V$ and $V_-$ are birational.
We conclude that $c \in \mathcal C_{p}^{bir} \subseteq \mathcal C_{k}^{bir}$.
\end{proof}

\subsection{Losev-Manin's moduli $3$-space}\label{toric -ex-section}
Strong duality is a rare property even for toric varieties, and it is not preserved under small modifications for dimension three or higher. 
 An explicit example of a toric threefold for which strong duality does not hold was proposed by Payne in \cite[Example 1]{Payne}.

In this section we shall discuss another interesting smooth threefold for which strong duality holds, i.e. such that $\mathcal{D}_1^\vee=\mathcal{C}_1$, but for which it fails, i.e. $\mathcal{D}_1^\vee\supsetneq\mathcal{C}_1$, as soon as we flop one or more curves. 
Let $X$ be the $3$-dimensional Losev-Manin moduli space  introduced in \cite{LM} as a toric compactification of the moduli space of rational curves with six marked points. 
It can be interpreted as the blow-up of $\PP^3$ at its four coordinate points and, subsequently, along the strict transforms of the six coordinate lines. 
We denote by $H$ the class of the pull-back on $X$ of a hyperplane of $\PP^3$, by $E_i,\ i=1,\dots,4$, the strict transforms of the exceptional divisors of the points and by $E_{ij}, \ 1\le i< j \le 4$, the exceptional divisors of the lines, which together generate the Picard group of $X$. 
Similarly, we denote by $h=H^2, e_i=-E_i^2, e_{ij}=HE_{ij}=E_iE_{ij}$ the classes of a general line, a general line in $E_i$ and a vertical fibre of $E_{ij}$ respectively: they generate the Chow group of $1-$cycles.
The matrix of the intersection pairing  $\N^1(X)_\mathbb{R}\times \N_1(X)_\mathbb{R}\to \mathbb{R}$,   with respect to the above bases, is diagonal and defined by:
$$
H\cdot h=1, \ E_i \cdot e_i=E_{ij}\cdot e_{ij}=-1.
$$

We denote by $E_{ijk}=H-E_i-E_j-E_k-E_{ij}-E_{ik}-E_{jk}$ the strict transform of a fixed hyperplane of $\PP^3$ spanned by three coordinate points. 
The following divisors generate the extremal rays of the effective cone $\mathcal{D}_0$ of $X$: 
$$E_i, E_{ij}, E_{ijk},$$
for all distinct $i,j,k$.
Since the above classes correspond to the torus invariant prime divisors of $X$, they give rise to a set of generators of the Cox ring (see \cite[Section 2.1.3]{ADHL}).

Now, using \cite[Prop. 3.3.2.3]{ADHL}, 
we can describe the movable cone of $X$ via its defining inequalities. In other terms, one can show that the extremal rays of the dual cone $\mathcal{D}_1^\vee$ are generated by the following effective curve classes, up to permutations of indices:
\begin{itemize}
\item curves sweeping out $E_{i}$: 
$e_i-e_{ij}$,
and
$2e_i-e_{ij}-e_{ik}-e_{il}$, 
\item curves sweeping out $E_{ij}$: $e_{ij}$ and $h-e_i-e_j+e_{ij}$,
\item curves sweeping out $E_{ijk}$: 
$h-e_i-e_{jk}$ and
$h-e_{ij}-e_{ik}-e_{jk}$.
\end{itemize}
This computation can be checked using any computer software that can compute cone duals, such as \emph{Magma} \cite{Magma}. 
While we are guaranteed that the effective cones $\mathcal D_0$ and the Mori cone $\mathcal C_0$ are dual to each other for all {small $\mathbb{Q}-$ factorial modifications of $X$, the same need not happen for the movable cone $\mathcal D_1$ and the cone of $1$-moving curves $\mathcal C_1$; whether or not it does, it depends on the model. In fact we will now discuss a small modification of $X$ for which $\mathcal D_1$ and  $\mathcal C_1$ are not dual cones.

In order to do this, it is convenient to describe $X$, and its small modifications, using the language of toric varieties. For a general reference on toric varieties, see 
\cite{toric-book} or \cite{Fulton}.
Since $X$ is obtained from $\PP^3$ by the iterated blow-up of torus invariant subvarieties, then $X$ is toric.  Let $N\cong\mathbb{Z}^3$ be a three-dimensional lattice; the rays of $\Sigma_X$ are generated  by the following primitive vectors: 
\begin{align*}  
& v_1=(1,1,1), v_2=(0,0,-1), v_3=(0,-1,0), v_4=(-1,0,0), & \\
& v_{12}=(1,1,0), v_{13}=(1,0,1), v_{14}=(0,1,1), & \\ 
& v_{23}=(0,-1,-1), v_{24}=(-1,0,-1), v_{34}=(-1,-1,0), & \\
& v_{123}=(1,0,0), v_{124}=(0,1,0), v_{134}=(0,0,1), v_{234}=(-1,-1,-1).&
\end{align*}
The ray generated by the vector $v_I$, which we will denote with $\rho_I$, corresponds to the boundary divisor $E_I$, for every index set $I\subset\{1,\dots,4\}$ of cardinality $1\le |I|\le 3$.
The  set of maximal cones, which are all simplicial, is the following: 
$$\textrm{Cone}(v_i,v_{ij},v_{ijk}),  \forall\{i\}\subset\{i,j\}\subset \{i,j,k\}\subset\{1,\dots,4\}.$$
The fan of $X$ is dual to a $3$-dimensional permutohedron, which is a polytope with $8$ hexagonal faces (corresponding to the boundary divisors $E_{ijk}$ and $E_i$) and six quadrilateral faces (corresponding to the $E_{ij}$'s). 
Figure \ref{figureLM} displays a section of a portion of the fan of $X$ lying in the plane of $N\otimes_{\mathbb{Z}}\mathbb{R}\cong \mathbb{R}^3$ of equation $x+y+z=a$, with $a>0$.
\begin{figure}[h]
\vskip-4em
\includegraphics[scale=0.85]{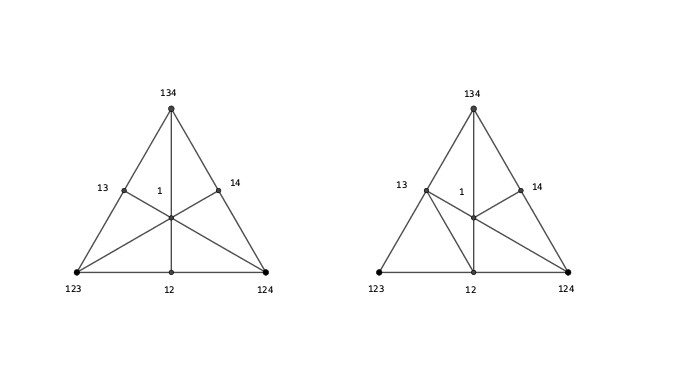}\vskip-5em
\caption{A section of a portion of the fan $\Sigma_X$ and of $\Sigma_{Y}$}
\label{figureLM}
\end{figure}
Consider the cone $\textrm{Cone}(v_{1},v_{123})$, which is a common two dimensional face of $\textrm{Cone}(v_1,v_{12},v_{123})$ and $\textrm{Cone}(v_1,v_{13},v_{123})$. The corresponding invariant subvariety is a floppable curve whose numerical class is $e_1-e_{12}-e_{13}$ and which corresponds to a fixed line inside $E_1$. Let $X\dashrightarrow Y$ be the flop of such curve.
The fan $\Sigma_Y$ differs from $\Sigma_X$ only within the quadrilateral cone generated by $v_{1}, v_{12}, v_{13}, v_{123}$ and, in particular, the image 
  $-e_1+e_{12}+e_{13}$ corresponds to the cone $\textrm{Cone}(v_{12},v_{13})\in \Sigma_Y$.

We argue that on $Y$ strong duality does not hold.
Consider the class $\gamma=e_1-e_{14}$ which is $1-$moving on $X$. The intersections  $\gamma\cdot E_1=-1$ and $\gamma\cdot E_{123}=1$ imply that every representative of $\gamma$ is contained in $E_1$ and it intersects the hyperplane $E_{123}$, where by abuse of notation we use the same symbols for cycles on $X$ and their birational images on $Y$. Since the two divisors  are disjoint in $Y$, no irreducible curve with class $\gamma$ exists on $Y$. In particular $\gamma\notin \mathcal{C}_1(Y)$; thus strong duality does not hold for $Y$, that is $\mathcal{C}_1(Y)\subsetneq \mathcal{D}_1^\vee$.

\section{Strong duality for blow ups of $\PP^n$}\label{section-strongduality-P^n}

In this section we consider $X=X^n_s$, i.e.\ the blow-up of $\PP^n$ at $s$ points in general position, for the pairs $(n,s)$ for which $X$ is a Mori dream space. 
These varieties are classified in \cite{Mukai04, Mukai05, CT}
and they fit in the following list:
 \begin{itemize}
 \item $X^2_s$ for $s\le 8$; 
 \item  $X^3_s$ for $s\le 7$; 
 \item  $X^4_s$ for $s\le 8$;
  \item $X^n_s$ for $n\ge 5$ and $s\le n+3$.
  \end{itemize}

The main result of this section is the proof of  Theorem \ref{intro-thm2}. More precisely, we show that the cone of $k$-moving curves is dual to the cone of divisors that are ample in codimension $k$,  for every Mori dream $X^n_s$. 
In the surface case $n=2$ there is nothing to prove.
The case $X^n_{s}$ with $s\le n+3$ and $n\ge3$,  will be treated in Section \ref{section-n+3}, while $X^3_7$ and $X^4_8$ will be considered in Section \ref{section-n+4}.

\subsection{Blow up of $\PP^n$ in $s\le n+3$ general points}\label{section-n+3}

Let $X=X_{s}^n$ be the
blow up of $\PP^n$ at $s\le n+3$ general points and assume that 
$n\ge3$.
We recall some notation from \cite{BraDumPos3}.
For any subset $I\subset\{1,\ldots,s\}$ we denote by $L_I$ the strict transform of the linear span of the points indexed by $I$. For $s=n+3$, we denote 
by $C$  the strict transform of the unique rational normal curve of degree $n$ passing through all points  and by $\sigma_t$ the strict transform of its $t$-secant variety, where we set $\sigma_1:=C$.
Let $J(L_I,\sigma_t)$ denote the strict transform of the join of the linear span of the points indexed by $I$ and the $t$-secant variety of the rational normal curve.
We make the following identifications: $J(L_I,\sigma_t)=\emptyset$ if $t=|I|=0$, $J(L_I,\sigma_t)=E_i$ if $t=0$ and $I=\{i\}$,  $J(L_I,\sigma_t)=L_I$ if $t=0$ and $|I|>1$, and $J(L_I,\sigma_t)=\sigma_t$ if $|I|=0$. 
It is easy to see that the join $J(L_I,\sigma_t)$ has dimension $2t+|I|-1$, unless $J(L_I,\sigma_t)=E_i$.

For $s\le n+3$, the 
inequalities cutting out the effective cone of $X^n_s$  were obtained in \cite{BraDumPos3} and a formula for the dimension of all linear systems of divisors on $X^n_s$ was proved in \cite{LPS}.
Moreover, by  \cite[Lemma 4.1]{BraDumPos3} and \cite[Proposition 4.2]{DumPos}, if $H$ denotes the general hyperplane class in $X^n_s$ and $E_1,\dots,E_{n+3}$ denote the exceptional divisors, then for an effective divisor
\begin{equation}\label{divisor on X^n_s}
D=dH-\sum_{i=1}^{n+3} m_iE_i,\end{equation}
the following integer 
\begin{equation} \label{k_J}
     \kappa_{J(L_I,\sigma_t)} = - (|I| + (n+1)t-1)d + \sum_{i\in I} (t+1)m_i + \sum_{i \notin I} t m_i, 
\end{equation}
 when positive, is the exact multiplicity of containment of  the join $J(L_I,\sigma_t)$ in the base locus of  $|D|$.
Notice that since equation \eqref{k_J} is linear in $d$ and in the $m_i$'s, then if the join $J(L_I,\sigma_t)$
is in the base locus of a linear system, it is a component of its stable base locus.

Formulas \eqref{k_J} induce the Mori chamber decomposition of the effective cone of $X^n_s$, already studied by Mukai. This turns out to coincide with its stable base locus decomposition.
This is the content of the following result that already appeared in \cite[Theorem 2.60]{BCP},
but that we include here for the sake of completeness. 

\begin{theorem} \label{n+3: SBLD=MCD}
For $X=X_{n+3}^n$, the Mori chamber decomposition and the stable base locus decomposition of the effective cone  coincide 
and they are induced by the following hyperplane arrangement 
\begin{equation}\label{hyperplane-arrangement-n+3}
\{\kappa_{J(L_I,\sigma_t)}=0: 0\le t \le \frac{n}{2}, 0\le |I|\le n-2t\},
\end{equation} 
where $\kappa_{J(L_I,\sigma_t)}$ is defined in \eqref{k_J}.
\end{theorem}
\begin{proof}
Notice that the hyperplane arrangement \eqref{hyperplane-arrangement-n+3} 
is in correspondence with the set
of subvarieties 
\begin{equation}\label{set-of-joins-n+3}
\{J(L_I,\sigma_t): 0\le t \le \frac{n}{2}, 0\le |I|\le n-2t \}.
\end{equation}
More precisely, \eqref{hyperplane-arrangement-n+3}
induces 
a chamber decomposition of the effective cone of divisors of $X$, such that the divisors in the interior of any given chamber contain the same joins in the stable base locus.
Such decomposition is refined by the stable base locus decomposition of the effective cone, because a priori effective divisors might have stable base locus that is not of  the form $J(L_I,\sigma_t)$. This turns out not to be the case for $X^n_{n+3}$.  

As explained in \cite[Section 4.3]{Araujo-Casagrande}, it follows from the work of Mukai (\cite{Mukai05}) that the Mori chamber decomposition of the effective cone of $X$
is given by  \eqref{hyperplane-arrangement-n+3}.  
Since the Mori chamber decomposition refines 
 the stable base locus decomposition, which is a further refinement of \eqref{hyperplane-arrangement-n+3}, this concludes the proof.
\end{proof}

 Recall that for a Mori dream space we have $\mathcal D_k= D_k$, by Thorem \ref{MCD-cones-divisors}.
 Using Theorem \ref{n+3: SBLD=MCD}, we can describe the cones $D_k$ of $X^n_{s}$ $s\le n+3$, with their inequalities. 

\begin{proposition} \label{n+3: D_k}
For any integer $k$ with $1\le k \le n-1$. The  cone  $\mathcal D_k$ of  divisors on $X^n_{n+3}$ that are ample in codimension $k$
 is cut out by the following inequalities:
\begin{enumerate}
\item[(\RN 1)] $0\le d$ and $m_i \le d $, for every $i=1, \ldots, n+3$;
\item[(\RN 2)] $-m_i + \sum_{j=1}^{n+3} m_j \le nd$, for every $i=1, \ldots, n+3$;
\item[(\RN 3$_k$)] $\kappa_{J(L_I,\sigma_t)}\le 0$, for every $0\leq t \le n/2$ and $|I|=n-2t-k+1$;
\item[(\RN 4)] $0 \le m_i$, for every $i=1, \ldots, n+3$.
\end{enumerate}
\end{proposition}

\begin{proof}
For $k=1$ the statement is proved in \cite[Theorem 5.3]{BraDumPos3}.

We proceed by induction on $k$. We want to prove that $\mathcal D_k$ is defined by the inequalities (\RN 1), (\RN 2), (\RN 3$_k$) and (\RN 4), for $k\ge 2$. We know that $\mathcal D_k \subseteq \mathcal D_{k-1}$. Hence, by the induction hypothesis, we have that $\mathcal D_k$ is contained in the cone cut out by the  inequalities (\RN 1), (\RN 2), (\RN 3$_{k-1}$) and (\RN 4). Notice that the only inequalities to be added to the list are the ones excluding divisors with stable base locus of dimension exactly $n-k$. By Theorem \ref{n+3: SBLD=MCD}, these are given by $\kappa_{J(L_I,\sigma_t)}\le 0$ for all joins $J(L_I,\sigma_t)$ of dimension $n-k$. Recalling that the dimension of any such join is $|I|+2t-1$, we get that these inequalities are precisely the ones of (\RN 3$_k$). We conclude by noticing that, by definition of $\kappa_{J(L_I,\sigma_t)}$ and the effectivity conditions (\RN 1) and (\RN 2), the inequalities (\RN 3$_k$) imply the ones of (\RN 3$_{k-1}$), so they become redundant and can be discarded.
\end{proof}

\begin{remark}\label{eff-cone-n+3}
The effective cone $\mathcal D_0$ of $X^n_{n+3}$ is described in
\cite[Theorem 5.1]{BraDumPos3}. The defining inequalities in this case are only (\RN 1), (\RN 2), (\RN 3$_{k=0}$).
\end{remark}

\begin{corollary}  \label{s<n+3: D_k}
 For $X=X_{s}^n$ with $1\le s\le n+2$,
  the cone  $\mathcal D_k$ is defined by the inequalities from Proposition \ref{n+3: D_k} as follows:
 \begin{enumerate} 
\item if $s=n+2$,
by the inequalities (\RN 1), (\RN 2), (\RN 3)$_k$ (with $t=0$) and $(\RN 4)$;
\item if $s\le n+1$, then 
\begin{enumerate}
\item if $k\ge n-s+1$, by the inequalities (\RN 1), (\RN 3)$_k$ (with $t=0$) and $(\RN 4)$,
\item if $k< n-s+1$, then we have $\mathcal D_k=\mathcal D_{k+1}$.
\end{enumerate}
\end{enumerate}
\end{corollary}
\begin{proof}
The statement follows from Proposition \ref{n+3: D_k} by using the natural isomorphism between the N\'eron-Severi space of $X=X^n_s$ and the linear subspace of $\N^1(X^n_{n+3})$ defined by the equations $m_i=0$, for $s+1\le i \le n+3$.
\end{proof}

With a description of the cones $\mathcal D_k$ in $X_s^n$ via their inequalities at hand, it is immediate to compute their dual cones with respect to the standard intersection pairing. Recall that $\N_1(X_s^n)$ is generated by the class of a general line in $X_s^n$, which we denote by $h$, along with the classes $e_i$, for every $i=1,\dots,n+3$, of a general line in the exceptional divisor $E_i$. Thus,  the intersection product of a  divisor $D \in \N^1(X_s^n)_\mathbb R$ of the form  \eqref{divisor on X^n_s} and a curve of class
\begin{equation}\label{curve on X^n_s}
c=\delta h - \sum_{i=1}^s\mu_i e_i \in \N_1(X_s^n)_\mathbb R
\end{equation} 
is computed  by
\begin{equation} \label{n+3: int pair} D\cdot c= d\delta  - \sum m_i \mu_i. \end{equation}
 In particular, to every hyperplane in $\N^1(X_s^n)_\mathbb R$, there corresponds a curve class in $\N_1(X_s^n)_\mathbb R$.

We will show that for $X^n_s$, with $s\le n+3$, the dual of $\mathcal D_k$ is precisely $\mathcal C_k$ for every $0\leq k \le n-1$, that is:  strong duality holds for every $k$. 
We need to introduce another ingredient before we can state and prove the result:
the Weyl group action on $\N_1(X_s^n)_\mathbb R$, which is described in detail in \cite[Proposition 2.5]{DM2}. For an index set $\Gamma \subseteq \{1, \ldots, s\}$ of length $n+1$,  the standard Cremona transformation ${\rm Cr}_\Gamma:\PP^n\dasharrow\PP^n$ based at the points of $\PP^n$ parametrized by $\Gamma$, lifts to an isomorphism of  $\N_1(X_s^n)_\mathbb R$ sending the curve class $c\in \N_1(X_s^n)_\mathbb R$ of the form \eqref{curve on X^n_s} to the following curve class:

\begin{equation} \label{Cremona curves}
{\rm Cr}_\Gamma (c) = (\delta - (n-1)a_\Gamma) h - \sum_{j \in \Gamma}(\mu_j - a_\Gamma)e_j - \sum_{i \notin \Gamma}\mu_i e_i,
\end{equation}
where we set
\begin{equation*}  a_\Gamma := \sum_{j \in \Gamma} \mu_j - \delta.
\end{equation*}
A composition of Cremona transformations is called a Weyl transformation and we say that two curve classes $c$ and $c'$ are in the same Weyl orbit if there is a Weyl transformation $\phi$ such that $\phi(c)=c'$. 

We are now in position to prove the following theorem.
\begin{theorem}[Strong duality theorem for $X^n_{n+3}$]\label{strong-duality-thm-n+3}
Let $1\le s \le n+3$ and $X=X^n_s$, then 
 $$\mathcal{C}_k=\mathcal D_k^\vee.$$
\end{theorem}

\begin{proof} 
Using the intersection pairing \eqref{n+3: int pair}, for every $k$, the inequalities of Proposition \ref{n+3: D_k} correspond to the following effective curve classes which generate the dual cone $\mathcal D_k^\vee$:
\begin{itemize}
\item[(\RN 1)] $h$ and $h-e_i$, for every $i=1, \ldots, n+3$;
\item[(\RN 2)] $nh-\sum_{j=1}^{n+3}e_j + e_i$, for every $i=1, \ldots, n+3$;
\item[(\RN 3$_k$)] $ (|I| + (n+1)t-1)h - \sum_{i\in I} (t+1)e_i - \sum_{i \notin I} t e_i$,  for every $0\leq t \le n/2$ and $|I|=n-2t-k+1$;
\item[(\RN 4)] $e_i$, for every $i=1, \ldots, n+3$.
\end{itemize}
Since the inclusion $\mathcal C_k \subseteq \mathcal D_k^\vee$ is always satisfied, in  order to prove the strong duality statement, we only need to show that each one of the above curve classes belongs to $\mathcal C_k$.

From now on we will assume that $s=n+3$. For smaller $s$ the proof is similar and uses Corollary \ref{s<n+3: D_k}, so we leave the details to the reader.

First of all notice that the curves classes (\RN 1) and (\RN 2) are dual to 
codimension one faces of the effective cone, by Remark \ref{eff-cone-n+3}.
Hence, by duality, they are in $\mathcal C_0$  which is contained in $\mathcal C_k$ for every $k$.

The curves in the classes (\RN 4) sweep out the exceptional divisors, so they belong to $\mathcal C_1$, which is contained in $\mathcal C_k$ for every $1\le k \le n-1$.

It remains to deal with the curve classes (\RN 3$_k$). Given a pair $(t, I)$ satisfying the bounds in (\RN 3$_k$), the  curve class 
\begin{equation} \label{c: join} c_{I,t}=  (|I| + (n+1)t-1)h - \sum_{i\in I} (t+1)e_i - \sum_{i \notin I} t e_i \end{equation}
is dual to the wall $\kappa_{J(L_I,\sigma_t)}=0$.  
This wall separates the divisors which do not contain the join $J(L_I,\sigma_t)$ in their stable base locus  from those that do. By Theorem \ref{PayneChoi}, we have that  
there is a family of curves of class $c_{I,t}$ that sweep out a subvariety of some small modification of $X$ that is birational to $J(L_{I},\sigma_t)\subset X$. 
 We claim that such family lives on $X$ and it sweeps out $J(L_{I},\sigma_t)$. 
We now prove the claim. 
Assume that $t=0$. In this case the class of $c_{I,0}$ reads as follows:
\begin{equation} \label{curves sweep linear}
 (n-k)h - \sum_{|I|=n-k+1} e_i.
\end{equation}
This curve class has irreducible representatives which sweep out the linear cycle $L_{I}$, see Example \ref{lin curve}. 

 Now, for any $t> 0$ we show that the curve class \eqref{c: join}  is in the Weyl orbit of a curve class of type \eqref{curves sweep linear}. 
Assume the statement is true for any pair $(t-1, I')$ satisfying the conditions in (\RN 3$_k$), and let $c_{I,t}$ be a curve class for a pair $(t, I)$ as in \eqref{c: join}. After relabelling the indexes if necessary, we may assume that $I = \{1, \ldots , r_t \}$ where $r_t=n-2t-k+1$. Now take $\Gamma = \{1, \ldots, r_t\}\cup \{r_t+3, \ldots, n+3\}$, we have that
$${\rm Cr}_\Gamma(c_{I,t}) = (r_t+2+(n+1)(t-1)-1)h - \sum_{i=1}^{r_t+2} te_i - \sum_{i=r_t+3}^{n+3} (t-1)e_i, $$
which is a curve class as in (\RN 3$_k$) for the pair $(t-1, I'=I\cup\{r_t+1, r_t+2\})$. By the induction hypothesis, there is a Weyl transformation $\phi'$ taking ${\rm Cr}_\Gamma(c_{I,t})$ to a curve with class as in \eqref{curves sweep linear}. Thus, taking $\phi = \phi' \circ {\rm Cr}_\Gamma$, we see that $c_{I,t}$ is in the Weyl orbit of some curve class with irreducible representatives sweeping out a linear cycle of dimension equal to the dimension of $J(L_{I}, \sigma_t)$. This proves that $c_{I,t}$ has irreducible representatives on $X$, which are the images of the irreducible representatives of \eqref{curves sweep linear}, {and they sweep out $J(L_I,\sigma_t)$ on $X$.}
Hence, $c$ belongs to $\mathcal C_k$. 
\end{proof}

\begin{remark}
We recall that an $(i)-$curve class on a smooth $n$-dimensional variety $X$ is defined as a smooth irreducible rational curve with normal bundle isomorphic to $\oplus_{i=1}^{n-1}\mathcal{O}(i)$, for $i\in \{-1,0,1\}$, see \cite[Definition 1.1]{DM2} and \cite{DM3}.
The extremal rays of $\mathcal{C}_0$
are examples of $(1)-$curves
and $(0)-$curves. 
The extremal rays of the Mori cone $\mathcal{C}_{n-1}$ of type $h-e_i-e_j$ and (for $s= n+3$) the rational normal curve class $C$
are all the $(-1)-$curves of $X^n_s$.
\end{remark}

\subsection{Blow up of $\PP^4$ in $8$ points and of $\PP^3$ in $7$ points}\label{section-n+4}
We now prove strong duality for the  two remaining cases:  $X^3_7$ and $X^4_8$.

We first recall the definition of Weyl cycles in $X^n_s$ introduced in \cite{BDP-Ciro}. 
A {\it Weyl divisor} is an effective divisor $D\in\Pic(X^n_s)$ which belongs to the Weyl orbit of an exceptional divisor $E_i$.
A {\it Weyl cycle} of codimension $r$ is
an irreducible component of the intersection of pairwise orthogonal Weyl divisors, where orthogonality is taken with respect to the Dolgachev-Mukai pairing.

\begin{theorem}\label{same decomposition3-4}
For $X=X_{7}^3$ or $X=X_{8}^4$ the Mori chamber decomposition and the stable base locus decomposition  coincide. In particular they are induced by the Weyl cycles of $X$.
\end{theorem}

\begin{proof}
The proof is similar to that of Theorem \ref{n+3: SBLD=MCD} and it uses two ingredients: on the one hand the classification of Weyl cycles and their corresponding base locus lemmata done in \cite{BDP-Ciro}; on the other hand the description of the Mori chamber decomposition. For the threefold $X^3_7$ the latter is easy to obtain and left to the reader (there are only floppable curves and contracting divisors). On the other hand, obtaining the Mori chamber decomposition for the fourfold $X^4_8$ is not trivial: it was done by Mukai \cite{Mukai05}  and by Casagrande, Codogni and Fanelli \cite{CCF} by means of the Gale duality between $X^4_8$ and the del Pezzo surface $S:=X^2_8$, a projective plane blown up in eight general points.

In \cite{BDP-Ciro} (see also \cite{DM1} for an alternative approach) the Weyl cycles on $X^3_7$ and $X^4_8$ are classified and explicitly described. 
For $X^3_7$ these are $(-1)-$curves and Weyl divisors, (see \cite[Section 4]{BDP-Ciro}),
for $X_8^4$ we further have five types of Weyl surfaces  (see \cite[Section 5]{BDP-Ciro} for details): they are the strict transforms of the following surfaces of $\PP^4$:
\begin{itemize}
\item $S^1_{i,j,k}$ a plane through three points,
\item $S^{3}_{i, \widehat{j}}=J(C_{\widehat{j}},p_i)$, a pointed cone over the rational normal curve passing through seven points $p_1,\dots, \widehat{p_j},\dots,p_8$, 
\item $S^{6}_{i,j,k}$ a sextic surface with five triple points and simple at three points,
\item $S^{10}_{i,j}$ a degree$-10$ surface with two sextuple points and five triple points,
\item $S^{15}_{i}$ a degree$-15$ surface with seven sextuple points and a triple point.
\end{itemize}

Associating to  
each of the above Weyl cycles a hyperplane in the N\'eron-Severi space, we define a hyperplane arrangement which induces 
a chamber decomposition (that we may call {\it Weyl chamber decomposition}) of the effective cone of divisors of $X$, such that the divisors in any chamber contain the same Weyl cycles in their stable base locus.

More precisely: the hyperplane associated to a Weil curve $C$ is defined by the equation $D\cdot C=0$; the hyperplane associated to a Weyl divisor $E$ is defined by $\langle D, E\rangle=0$, where $\langle\cdot,\cdot\rangle$ is the Dolgachev-Mukai pairing.
The hyperplanes associated to the five surfaces have equations $D\cdot \mu=0$, where $\mu$ is one of the following curve classes:
\begin{itemize}
\item $\mu^1_{i,j,k}=2h-e_i-e_j-e_k$,
\item $\mu^{3}_{i, \widehat{j}}=5h-\sum_{k=1}^8e_k-e_i+e_j$,
\item $\mu^{6}_{i,j,k}=8h-2\sum_{l=1}^8e_l+e_i+e_j+e_k$,
\item $\mu^{10}_{i,j}=11h-2\sum_{k=1}^8e_k-e_i-e_j$,
\item $\mu^{15}_{i}=14h-3\sum_{j=1}^8e_j+e_i$.
\end{itemize}

Now,  the Weyl chamber decomposition is refined by the stable base locus decomposition, since effective divisors may have stable base locus which is not of Weyl type. Moreover, recall that the stable base locus decomposition is, in turn, refined by the Mori chamber decomposition, described in \cite[Theorem 5.13]{CCF}. Therefore, since the Weyl and Mori decompositions agree, 
we conclude.
\end{proof}

\begin{theorem}\label{strong-duality-thm-X^4_8}
Let $X=X^3_7$ or 
$X=X^4_8$, then 
 $$\mathcal{C}_k=\mathcal D_k^\vee.$$
\end{theorem}
\begin{proof}
Following the proof of Theorem \ref{strong-duality-thm-n+3}, one can prove the duality statement for $\mathcal{C}_1$ and $\mathcal{C}_2$.  Since for the movable cone $\mathcal{D}_1$ the statement is easy, we just give a hint of the proof for the case of $\mathcal D_2(X^4_8)$.

 Using the notation of the proof of Theorem \ref{same decomposition3-4}, one first observes that $\mu^1_{1,2,3}$ belongs to $\mathcal C_2(X^4_8)$ (see also Example \ref{lin curve}) and then verifies that the remaining four curve types $\mu^1_{1,2,3}$, $\mu^{3}_{1, \widehat{8}}$,
$\mu^{6}_{1,2,3}$,
$\mu^{10}_{1,2}$ and
$\mu^{15}_{1}$
may be transformed to $\mu^1_{i,j,k}$ via the Weyl action, for some indices $i,j,k$.
The details are left to the reader.
\end{proof}


\bigskip



\begin{bibdiv}
\begin{biblist}

\bib{Araujo-Casagrande}{article}{
    AUTHOR = {Araujo, C.},
    AUTHOR = {Casagrande, C.},
     TITLE = {On the {F}ano variety of linear spaces contained in two
              odd-dimensional quadrics},
   JOURNAL = {Geom. Topol.},
    VOLUME = {21},
      YEAR = {2017},
    NUMBER = {5},
     PAGES = {3009--3045},
      ISSN = {1465-3060,1364-0380},
       DOI = {10.2140/gt.2017.21.3009},
       URL = {https://doi.org/10.2140/gt.2017.21.3009},
}

\bib{ADHL}{book}{
    AUTHOR = {Arzhantsev, I.},
    AUTHOR ={Derenthal, U.}, 
    AUTHOR = {Hausen, J.},
    AUTHOR = {Laface, A.},
     TITLE = {Cox rings},
    SERIES = {Cambridge Studies in Advanced Mathematics},
    VOLUME = {144},
 PUBLISHER = {Cambridge University Press, Cambridge},
      YEAR = {2015},
     PAGES = {viii+530},
      ISBN = {978-1-107-02462-5},
   MRCLASS = {14Cxx (14Jxx 14Lxx)},
  MRNUMBER = {3307753},
MRREVIEWER = {Alexandr V. Pukhlikov},
}
	
	
	
\bib{BCHM}{article}{
    AUTHOR = {Birkar, C.},
    AUTHOR={ Cascini, P.}, 
    AUTHOR={Hacon, C. D.},
     AUTHOR={McKernan, J.},
     TITLE = {Existence of minimal models for varieties of log general type},
   JOURNAL = {J. Amer. Math. Soc.},
    VOLUME = {23},
      YEAR = {2010},
    NUMBER = {2},
     PAGES = {405--468},
      ISSN = {0894-0347},
       URL = {https://doi-org.ezp.biblio.unitn.it/10.1090/S0894-0347-09-00649-3},
}

\bib{Magma}{article}{
   author={Bosma, W.},
   author={Cannon, John},
   author={Playoust, Catherine},
   title={The Magma algebra system. I. The user language},
   note={Computational algebra and number theory (London, 1993)},
   journal={J. Symbolic Comput.},
   volume={24},
   date={1997},
   number={3-4},
   pages={235--265},
   issn={0747-7171},
}


\bib{BDPP}{article}{
AUTHOR = {Boucksom, S.},
AUTHOR = {Demailly, J.},
AUTHOR = {P\u{a}un, M.},
AUTHOR = {Peternell, T.},
     TITLE = {The pseudo-effective cone of a compact {K}\"{a}hler manifold and
              varieties of negative {K}odaira dimension},
   JOURNAL = {J. Algebraic Geom.},
    VOLUME = {22},
      YEAR = {2013},
    NUMBER = {2},
     PAGES = {201--248},
      ISSN = {1056-3911},
       URL = {https://doi.org/10.1090/S1056-3911-2012-00574-8},
}

\bib{BraDumPos3}{article}{
    AUTHOR = {Brambilla, M. C.},
    AUTHOR = {Dumitrescu, O.},
    AUTHOR = {Postinghel, E.},
     TITLE = {On the effective cone of {$\Bbb P^n$} blown-up at {$n+3$}
              points},
   JOURNAL = {Exp. Math.},
    VOLUME = {25},
      YEAR = {2016},
    NUMBER = {4},
     PAGES = {452--465},
      ISSN = {1058-6458},
}

\bib{BDP-Ciro}{book}{
AUTHOR = {Brambilla, M. C.},
    AUTHOR = {Dumitrescu, O.},
    AUTHOR = {Postinghel, E.},
     TITLE = {Weyl cycles on the blow-up of $\PP^4$ at eight points},
     BOOKTITLE= {The art of doing algebraic geometry},
     PUBLISHER={Birkh\"auser Cham},
     SERIES = {The art of doing algebraic geometry, Trends in Mathematics},
     PAGES={1-21},
     YEAR = {2023},
}



\bib{Brion}{article}{
    AUTHOR = {Brion, M.},
     TITLE = {Vari\'{e}t\'{e}s sph\'{e}riques et th\'{e}orie de {M}ori},
   JOURNAL = {Duke Math. J.},
  FJOURNAL = {Duke Mathematical Journal},
    VOLUME = {72},
      YEAR = {1993},
    NUMBER = {2},
     PAGES = {369--404},
      ISSN = {0012-7094},
   MRCLASS = {14L30 (14M17)},
  MRNUMBER = {1248677},
MRREVIEWER = {Franz Pauer},
       URL = {https://doi-org.ezp.biblio.unitn.it/10.1215/S0012-7094-93-07213-4},
}
		

\bib{BOP-S}{article}{
    AUTHOR = {Broomhead, N.}, 
    AUTHOR=  {Ottem, J. C. }, 
     AUTHOR= {Prendergast-Smith, A.},
     TITLE = {Partially ample line bundles on toric varieties},
   JOURNAL = {Glasg. Math. J.},
  FJOURNAL = {Glasgow Mathematical Journal},
    VOLUME = {58},
      YEAR = {2016},
    NUMBER = {3},
     PAGES = {587--598},
      ISSN = {0017-0895},
   MRCLASS = {14M25 (14C20)},
  MRNUMBER = {3530488},
MRREVIEWER = {Shin-Yao Jow},
       URL = {https://doi.org/10.1017/S001708951500035X},
}
	

\bib{BCP}{book}{
    AUTHOR = {Bus\'e, L.},
    AUTHOR = {Catanese, F.},
    AUTHOR = {Postinghel, E.},
     TITLE = {Algebraic curves and surfaces},
          SUBTITLE = {A history of shapes},
    SERIES = {SISSA Springer Series},
    VOLUME={4},
 PUBLISHER = {Springer Cham},
      YEAR = {2023},
     PAGES = {XIV+205},
URL={https://link.springer.com/book/10.1007/978-3-031-24151-2},
}







\bib{CCF}{article}{
AUTHOR = {Casagrande, C.},
AUTHOR = {Codogni, G.},
AUTHOR = {Fanelli, A.},
     TITLE = {The blow-up of {$\Bbb {P}^4$} at 8 points and its {F}ano
              model, via vector bundles on a del {P}ezzo surface},
   JOURNAL = {Rev. Mat. Complut.},
  FJOURNAL = {Revista Matem\'{a}tica Complutense},
    VOLUME = {32},
      YEAR = {2019},
    NUMBER = {2},
     PAGES = {475--529},
      ISSN = {1139-1138},
   MRCLASS = {14J60 (14E30 14J35 14J45)},
  MRNUMBER = {3942925},
MRREVIEWER = {Sarbeswar Pal},
       URL = {https://doi.org/10.1007/s13163-018-0282-5},
}


\bib{Castelnuovo91}{article}{
    AUTHOR = {Castelnuovo, G.},
         TITLE = {Ricerche generali sopra i sistemi lineari di curve piane},
  JOURNAL = {Mem. Accad. Sci. Torino},
    VOLUME = {II 42},
      YEAR = {1891},
     PAGES = {3--43},
}


\bib{CT}{article}{
    AUTHOR = {Castravet, A. M.},
    AUTHOR = {Tevelev, J.},
     TITLE = {Hilbert's 14th problem and {C}ox rings},
   JOURNAL = {Compos. Math.},
  FJOURNAL = {Compositio Mathematica},
    VOLUME = {142},
      YEAR = {2006},
    NUMBER = {6},
     PAGES = {1479--1498},
      ISSN = {0010-437X},
   MRCLASS = {14L30 (13A50 14C22 14M20)},
  MRNUMBER = {2278756},
MRREVIEWER = {Matthias Meulien},
       URL = {https://doi-org.ezp.biblio.unitn.it/10.1112/S0010437X06002284},
}


	



\bib{ChoiCONES}{article}{
    AUTHOR = {Choi, S. R.},
     TITLE = {Duality of the cones of divisors and curves},
   JOURNAL = {Math. Res. Lett.},
    VOLUME = {19},
      YEAR = {2012},
    NUMBER = {2},
     PAGES = {403--416},
      ISSN = {1073-2780},
}

\bib{Choi14}{article}{
        AUTHOR = {Choi, S. R.},
     TITLE = {On partially ample adjoint divisors},
   JOURNAL = {J. Pure Appl. Algebra},
    VOLUME = {218},
      YEAR = {2014},
    NUMBER = {7},
     PAGES = {1171--1178},
      ISSN = {0022-4049},
}

\bib{Choi-Gongyo}{article}{
    AUTHOR = {Choi, S. R.},
    AUTHOR = {Gongyo, Y.},
     TITLE = {On a generalized Batyrev's cone conjecture},
   JOURNAL = {Math. Z.},
    VOLUME = {300},
      YEAR = {2022},
    NUMBER = {2},
     PAGES = {1319--1334},
      ISSN = {0025-5874,1432-1823},
}

\bib{toric-book}{book}{
    AUTHOR = {Cox, D. A.},
    AUTHOR= {Little, J. B.}, 
    AUTHOR= {Schenck, H. K.},
     TITLE = {Toric varieties},
    SERIES = {Graduate Studies in Mathematics},
    VOLUME = {124},
 PUBLISHER = {American Mathematical Society, Providence, RI},
      YEAR = {2011},
     PAGES = {xxiv+841},
      ISBN = {978-0-8218-4819-7},
   MRCLASS = {14M25 (05A15 05E45 52B12)},
  MRNUMBER = {2810322},
MRREVIEWER = {Ivan Arzhantsev},
}

\bib{DumPos}{article}{
    AUTHOR = {Dumitrescu, O.},
    AUTHOR = {Postinghel, E.},
    TITLE = {Vanishing theorems for linearly obstructed divisors},
   JOURNAL = {J. Algebra},
  FJOURNAL = {Journal of Algebra},
    VOLUME = {477},
      YEAR = {2017},
     PAGES = {312--359},
      ISSN = {0021-8693},
   MRCLASS = {14C20 (14C17 14J70 14N05)},
  MRNUMBER = {3614155},
MRREVIEWER = {Jorge Caravantes},
       URL = {https://doi-org.ezp.biblio.unitn.it/10.1016/j.jalgebra.2017.01.006},
}



\bib{DM1}{book}{
	AUTHOR = {Dumitrescu, O.},
	AUTHOR = {Miranda, R.},
	TITLE = {Cremona Orbits in $\mathbb{P}^4$ and Applications},
      BOOKTITLE= {The art of doing algebraic geometry},
     PUBLISHER={Birkh\"auser Cham},
     SERIES = {The art of doing algebraic geometry, Trends in Mathematics},
	YEAR = {2023},
	PAGES = {161--185},
}



\bib{DM2}{article}{
	AUTHOR = {Dumitrescu, O.},
	AUTHOR = {Miranda, R.},
	TITLE = {On $(i)$ curves in blowups of $\mathbb{P}^r$},
	YEAR = {2021},
        URL = {https://arxiv.org/pdf/2104.14141.pdf},
	JOURNAL = {arXiv:2104.14141}}

\bib{DM3}{article}{
	AUTHOR = {Dumitrescu, O.},
	AUTHOR = {Miranda, R.},
	TITLE = {Coxeter theory for curves on blowups of $\mathbb{P}^r$},
	YEAR = {2022},
	URL = {https://arxiv.org/pdf/2205.13605.pdf},
	JOURNAL = {arXiv:2205.13605}}

\bib{DP-positivityI}{article}{
	AUTHOR = {Dumitrescu, O.},
	AUTHOR = {Postinghel, E.},
	TITLE = {Positivity of divisors on blown-up projective spaces, I},
	JOURNAL = { Ann. Sc. Norm. Super. Pisa Cl. Sci.},
        VOLUME = {24},
        PAGES = {599--618},
	YEAR = {2023},}



\bib{ELMNP}{article}{
    AUTHOR = {Ein, L.},
AUTHOR = {Lazarsfeld, R.},
AUTHOR = {Musta\c{t}\u{a}, M.},
AUTHOR = {Nakamaye, M.},
AUTHOR = {Popa, M.},
     TITLE = {Asymptotic invariants of base loci},
   JOURNAL = {Ann. Inst. Fourier (Grenoble)},
  FJOURNAL = {Universit\'{e} de Grenoble. Annales de l'Institut Fourier},
    VOLUME = {56},
      YEAR = {2006},
    NUMBER = {6},
     PAGES = {1701--1734},
      ISSN = {0373-0956},
   MRCLASS = {14C20 (14B05 14F17)},
  MRNUMBER = {2282673},
MRREVIEWER = {Tomasz Szemberg},
       URL = {http://aif.cedram.org/item?id=AIF_2006__56_6_1701_0},
}

\bib{Fulton}{book}{
    AUTHOR = {Fulton, W.},
     TITLE = {Introduction to toric varieties},
    SERIES = {Annals of Mathematics Studies},
    VOLUME = {131},
      NOTE = {The William H. Roever Lectures in Geometry},
 PUBLISHER = {Princeton University Press, Princeton, NJ},
      YEAR = {1993},
     PAGES = {xii+157},
      ISBN = {0-691-00049-2},
   MRCLASS = {14M25 (14-02 14J30)},
  MRNUMBER = {1234037},
MRREVIEWER = {T. Oda},
       URL = {https://doi.org/10.1515/9781400882526},
}
	
	

\bib{Hu-Keel}{article}{
    AUTHOR = {Hu, Y.},
   AUTHOR = {Keel, S.},
     TITLE = {Mori dream spaces and {GIT}},
      NOTE = {Dedicated to William Fulton on the occasion of his 60th
              birthday},
   JOURNAL = {Michigan Math. J.},
  FJOURNAL = {Michigan Mathematical Journal},
    VOLUME = {48},
      YEAR = {2000},
     PAGES = {331--348},
      ISSN = {0026-2285},
   MRCLASS = {14L24 (14E30)},
  MRNUMBER = {1786494},
MRREVIEWER = {P. E. Newstead},
       URL = {https://doi.org/10.1307/mmj/1030132722},
}


\bib{Huizenga}{incollection}{
    AUTHOR = {Huizenga, J.},
     TITLE = {Birational geometry of moduli spaces of sheaves and
              {B}ridgeland stability},
 BOOKTITLE = {Surveys on recent developments in algebraic geometry},
    SERIES = {Proc. Sympos. Pure Math.},
    VOLUME = {95},
     PAGES = {101--148},
 PUBLISHER = {Amer. Math. Soc., Providence, RI},
      YEAR = {2017},
   MRCLASS = {14J60 (14C05 14E30 14J29)},
  MRNUMBER = {3727498},
MRREVIEWER = {Carla Novelli},
       URL = {https://doi.org/10.1090/pspum/095/01639},
}
\bib{LMR}{article}{
AUTHOR = {Laface, A.},
AUTHOR = {Massarenti, A.},
AUTHOR = {Rischter, R.},
     TITLE = {On {M}ori chamber and stable base locus decompositions},
   JOURNAL = {Trans. Amer. Math. Soc.},
  FJOURNAL = {Transactions of the American Mathematical Society},
    VOLUME = {373},
      YEAR = {2020},
    NUMBER = {3},
     PAGES = {1667--1700},
      ISSN = {0002-9947},
   MRCLASS = {14E05 (14C20 14J45 14L10 14M15)},
  MRNUMBER = {4068278},
MRREVIEWER = {Sung Rak Choi},
       URL = {https://doi.org/10.1090/tran/7985},
}

\bib{LPS}{article}{
    AUTHOR = {Laface, A.},
	AUTHOR={Postinghel, E.},
	AUTHOR={Santana S\'{a}nchez, L. J.},
     TITLE = {On linear systems with multiple points on a rational normal
              curve},
   JOURNAL = {Linear Algebra Appl.},
  FJOURNAL = {Linear Algebra and its Applications},
    VOLUME = {657},
      YEAR = {2023},
     PAGES = {197--240},
      ISSN = {0024-3795},
   MRCLASS = {14C20 (14J17 14J70)},
  MRNUMBER = {4507643},
       URL = {https://doi.org/10.1016/j.laa.2022.10.023},
}

		
\bib{Laz1}{book}{
    AUTHOR = {Lazarsfeld, R.},
     TITLE = {Positivity in algebraic geometry. {I}},
    SERIES = {Ergebnisse der Mathematik und ihrer Grenzgebiete. 3. Folge. A
              Series of Modern Surveys in Mathematics [Results in
              Mathematics and Related Areas. 3rd Series. A Series of Modern
              Surveys in Mathematics]},
    VOLUME = {48},
      NOTE = {Classical setting: line bundles and linear series},
 PUBLISHER = {Springer-Verlag, Berlin},
      YEAR = {2004},
     PAGES = {xviii+387},
      ISBN = {3-540-22533-1},
       URL = {https://doi.org/10.1007/978-3-642-18808-4},
}


		
\bib{Laz2}{book}{
    AUTHOR = {Lazarsfeld, R.},
     TITLE = {Positivity in algebraic geometry. {II}},
    SERIES = {Ergebnisse der Mathematik und ihrer Grenzgebiete. 3. Folge. A
              Series of Modern Surveys in Mathematics [Results in
              Mathematics and Related Areas. 3rd Series. A Series of Modern
              Surveys in Mathematics]},
    VOLUME = {49},
      NOTE = {Positivity for vector bundles, and multiplier ideals},
 PUBLISHER = {Springer-Verlag, Berlin},
      YEAR = {2004},
     PAGES = {xviii+385},
     URL = {https://doi.org/10.1007/978-3-642-18808-4},
}




\bib{LM}{article}{
    AUTHOR = {Losev A. and Manin, Y.},
     TITLE = {New moduli spaces of pointed curves and pencils of flat
              connections},
      NOTE = {Dedicated to William Fulton on the occasion of his 60th
              birthday},
   JOURNAL = {Michigan Math. J.},
  FJOURNAL = {Michigan Mathematical Journal},
    VOLUME = {48},
      YEAR = {2000},
     PAGES = {443--472},
      ISSN = {0026-2285},
   MRCLASS = {14N35 (14H10 53D45)},
  MRNUMBER = {1786500},
MRREVIEWER = {Andrew Kresch},
       URL = {https://doi.org/10.1307/mmj/1030132728},
}
	

\bib{Mukai04}{incollection}{
    AUTHOR = {Mukai, S.},
     TITLE = {Geometric realization of {$T$}-shaped root systems and
              counterexamples to {H}ilbert's fourteenth problem},
 BOOKTITLE = {Algebraic transformation groups and algebraic varieties},
    SERIES = {Encyclopaedia Math. Sci.},
    VOLUME = {132},
     PAGES = {123--129},
 PUBLISHER = {Springer, Berlin},
      YEAR = {2004},
   MRCLASS = {13A50 (14C22 14M20 17B20 17B67)},
  MRNUMBER = {2090672},
MRREVIEWER = {Dmitry A. Timash\"{e}v},
       URL = {https://doi.org/10.1007/978-3-662-05652-3_7},
}

\bib{Mukai05}{article}{
    AUTHOR = {Mukai, S.},
     TITLE = {Finite generation of the Nagata invariant rings in A-D-E cases},
   JOURNAL = {RIMS Preprint},
  FJOURNAL = {},
    VOLUME = {},
      YEAR = {2005},
}


\bib{nagata}{article}{
    AUTHOR = {Nagata, M.},
     TITLE = {On the {$14$}th problem of {H}ilbert},
   JOURNAL = {S\={u}gaku},
  FJOURNAL = {Mathematical Society of Japan. S\={u}gaku (Mathematics)},
    VOLUME = {12},
      YEAR = {1960/61},
     PAGES = {203--209},
      ISSN = {0039-470X},
   MRCLASS = {14.08 (15.85)},
  MRNUMBER = {154867},
MRREVIEWER = {Y. Nakai},
}

\bib{okawa}{article}{ 
    AUTHOR = {Okawa, S.},
     TITLE = {On images of {M}ori dream spaces},
   JOURNAL = {Math. Ann.},
  FJOURNAL = {Mathematische Annalen},
    VOLUME = {364},
      YEAR = {2016},
    NUMBER = {3-4},
     PAGES = {1315--1342},
      ISSN = {0025-5831},
   MRCLASS = {14L24 (13A02 13A50)},
  MRNUMBER = {3466868},
MRREVIEWER = {Gergely B\'{e}rczi},
       URL = {https://doi-org.ezp.biblio.unitn.it/10.1007/s00208-015-1245-5},

}

\bib{Payne}{article}{
    AUTHOR = {Payne, S.},
     TITLE = {Stable base loci, movable curves, and small modifications, for
              toric varieties},
   JOURNAL = {Math. Z.},
  FJOURNAL = {Mathematische Zeitschrift},
    VOLUME = {253},
      YEAR = {2006},
    NUMBER = {2},
     PAGES = {421--431},
      ISSN = {0025-5874},
   MRCLASS = {14M25 (14C20)},
  MRNUMBER = {2218709},
MRREVIEWER = {Antonio Laface},
       URL = {https://doi.org/10.1007/s00209-005-0923-5},
}


\bib{Totaro}{article}{
    AUTHOR = {Totaro, B.},
     TITLE = {Line bundles with partially vanishing cohomology},
   JOURNAL = {J. Eur. Math. Soc. (JEMS)},
  FJOURNAL = {Journal of the European Mathematical Society (JEMS)},
    VOLUME = {15},
      YEAR = {2013},
    NUMBER = {3},
     PAGES = {731--754},
      ISSN = {1435-9855},
   MRCLASS = {14C20 (32L10)},
  MRNUMBER = {3085089},
MRREVIEWER = {V. A. Golubeva},
       URL = {https://doi.org/10.4171/JEMS/374},
}



\end{biblist}
\end{bibdiv}
\end{document}